\documentclass{amsart}
\usepackage{graphicx}
\usepackage{array}
\usepackage{epsfig}
\usepackage{epstopdf}
\usepackage{amsmath, amsfonts,amssymb,amsthm}
\usepackage{subfigure}
\theoremstyle{plain}
\newtheorem{Lem}{Lemma}[section]
\newtheorem{Thm}[Lem]{Theorem}
\newtheorem{Prop}[Lem]{Proposition}
\newtheorem{Cor}[Lem]{Corollary}

\theoremstyle{definition}
\newtheorem{Def}[Lem]{Definition}

\newcommand{\Rats}{\mathbb{Q}}
\newcommand{\Ints}{\mathbb{Z}}

\newcommand\Aut{\operatorname{Aut}}
\newcommand\stab{\operatorname{stab}}
\usepackage{kbordermatrix}

\begin{document}

\title{Matroid automorphisms of the root system $H_4$}
\author{Chencong Bao}
\address{Dept. of Mathematics\\
	Lafayette College\\
   	Easton, PA  18042-1781}

\author{Camila Freidman-Gerlicz}
\address{Dept. of Mathematics \& Computer Science \\
Claremont McKenna College\\
Claremont, CA 91711}
	
	\author{Gary Gordon}
\address{Dept. of Mathematics\\
	Lafayette College\\
   	Easton, PA  18042-1781}
 \email{gordong@lafayette.edu}
 
	\author{Peter McGrath}
\address{Dept. of Mathematics\\
	Lafayette College\\
   	Easton, PA  18042-1781}

 \author{Jessica Vega}
\address{Dept. of Mathematics\\
Loyola Marymount University \\
Los Angeles, CA 90045 }

 \thanks{Research supported by NSF grant DMS-055282.}


\begin{abstract}  We study the rank-4 linear matroid $M(H_4)$ associated with the 4-dimensional root system $H_4$.  This root system coincides with the vertices of the 600-cell, a 4-dimensional regular solid.  We determine the automorphism group of this matroid, showing half of the 14,400 automorphisms are geometric and half are not.  We prove this group is transitive on the flats of the matroid, and also prove  this group action is primitive.  We use the  incidence properties of the flats and the {\it orthoframes} of the matroid as a tool to understand these automorphisms, and interpret the flats geometrically.

\end{abstract}

\maketitle
  \markboth{\sc{C. Bao, C. Friedman-Gerlicz, G. Gordon, P. McGrath, and J. Vega}}{\sc{Matroid automorphisms of the root system $H_4$}}
  
\section{Introduction}  Regular polytopes in 4-dimensions are notoriously difficult to understand geometrically.  Coxeter's classic text  \cite{cox}  is an excellent resource, concentrating on both the metric properties  and the symmetry groups of regular polytopes.  Another approach to understanding these polytopes is through combinatorics; we use matroids to  model the linear dependence of a collection of vectors associated to the polytope.  That is the context for this paper, and we concentrate on the matroid associated with the 120-cell or the 600-cell, two dual 4-dimensional regular polytopes.

The connection between polytopes and matroids, or, more generally, between root systems and matroids, is as follows.  Given a finite set $S$ of vectors in $\mathbb{R}^n$ possessing a high degree of symmetry,  define the (linear) matroid $M(S)$ as the dependence matroid for the set $S$ over $\mathbb{R}$.  Then there should be a close relationship between the symmetry group of the original set $S$ ({\it geometric} symmetry) and the matroid automorphism group $\Aut(M(S))$ ({\it combinatorial} symmetry).  In particular, every geometric symmetry necessarily preserves the dependence structure of $S$, so gives rise to a matroid automorphism.
%
%For a simple example, let $S$ be the  root system associated with the regular $n$-gon.  Then the geometric symmetry group for $S$ is the dihedral group $D_n$.  The associated matroid $M(S)$ is the rank-2 uniform matroid $U_{2,n}$, and so $\Aut(M(S))\cong S_n$, the symmetric group.  In this case, the matroid automorphisms do not adequately capture the geometry of $S$.
%

The root system $H_4$ can be obtained by choosing the 120 vectors in $\mathbb{R}^4$ that form the vertices of the 600-cell.  These vectors come in 60 pairs, and each pair corresponds to a single point in the matroid.  Thus,  $M(H_4)$ is a rank-4 matroid on 60 points.  

This paper    generalizes and extends \cite{eg}.  In particular, we are interested in understanding the structure of the matroid automorphism group  $\Aut(M(H_4))$.  We show (Theorem~\ref{T:aut}) that $\Aut(M(H_4))$ contains {\it non-geometric} automorphisms in the sense that half of the 14,400 elements of $\Aut(M(H_4))$ do not arise from the Coxeter/Weyl group $W(H_4)$.   We also prove the automorphism group of $M(H_4)$ acts transitively on each class of flats of the matroid (Lemma~\ref{L:trans}), and that the action is primitive (Theorem~\ref{T:prim}).  A key tool for understanding the structure of the automorphisms is the incidence relation among the flats of $M(H_4)$ (Lemma~\ref{L:flatinc} and Proposition~\ref{P:planeintersect}).  This incidence structure allows us to compute the stabilizer of a point of the matroid (Lemma~\ref{L:stab}), a fact we need to understand the structure of the group.

\begin{figure}[htbp]
\begin{center}
\includegraphics[width=4in]{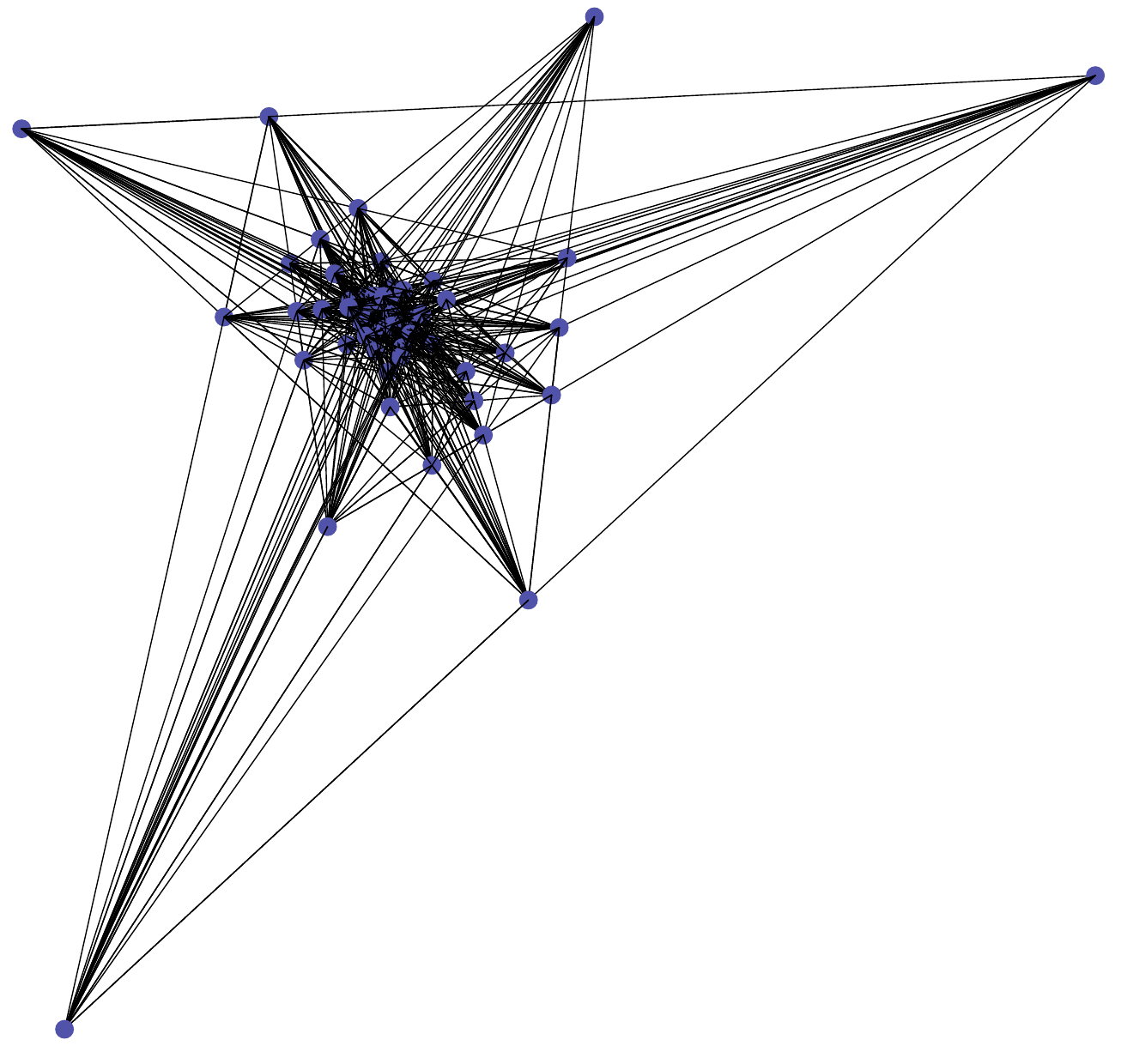}
\caption{A projection of the matroid $M(H_4)$.}
\label{F:alllines}
\end{center}
\end{figure}

 The connection between the geometric and combinatorial symmetry of certain root systems has been explored in \cite{dut,eg,fggp,fglp}.  In \cite{fglp}, matroid automorphism groups  are computed for the root systems $A_n, B_n$ and $D_n$, while \cite{eg} considers the root system $H_3$ associated with the icosahedron and \cite{fggp} examines the matroid associated with the root system $F_4$.  The general case is treated in \cite{dut}, where a computer program is employed to show that $\Aut (M(S))\cong G_S/W$ for all root systems $S$ {\it except} $F_4, H_3$ and $H_4$, where $G_S$ is the Coxeter/Weyl group associated with $S$ and $W$ is either the 2-element group $\Ints_2$ (when $G$ has {\it central inversion}) or $W$ is trivial (when $G$ does not have central inversion).  No attempt is made to understand the structure of these matroids in \cite{dut}, however.
 
 Other models for connecting geometric and combinatorial symmetry are possible, of course.  In particular, since each pair of vectors $\pm {\bf v}$ in a root system corresponds to a double point in the associated linear matroid, we could consider both vectors in the matroid.  This has the effect of doubling the number of automorphisms for each such pair; in our case, this increases the number of automorphisms by a factor of $2^{60}$.  Alternatively, we could, associate an {\it oriented} matroid with the root system.  This doubles the number of automorphisms considered here.  Another option is to consider a projective version of the  root system.  We point out, however, that all of these modifications differ from our treatment in transparent ways that do not change our understanding of the connection between the geometry and the combinatorics.
  
This paper is organized as follows:  The matroid $M(H_4)$ is defined as the column dependence matroid for a $4 \times 60$ matrix in Section~\ref{S:def}.  In Section~\ref{S:flats}, we describe the flats and {\it orthoframes} of the matroid and their incidence.  Orthoframes are special bases of the matroid, and they are important for understanding a certain kind of duality between points and 15-point planes.  This point-plane correspondence is made explicit  in Propositions~\ref{P:orthodata}(4), \ref{P:planeeqns}, and \ref{P:ptplanedual}, where it is interpreted combinatorially, algebraically and geometrically, respectively.  Orthoframes also allow us to reconstruct the matroid - Proposition~\ref{P:orthorecon}.

Section~\ref{S:aut} is the heart of this paper, concentrating on the structure of the matroid automorphisms.  We show that the stabilizer of a point $x$ is $\stab(x)\cong S_5\times \Ints_2$ (Lemma~\ref{L:stab}), then use this to show that $\Aut(M(H_4))$ acts transitively on flats (Lemma~\ref{L:trans}) and primitively on the matroid (Theorem~\ref{T:prim}).  This allows us to understand the structure of the group - Theorem~\ref{T:aut}.  We conclude (Section~\ref{S:geo}) with a few connections between the flats of the matroid and various classes of faces of the 120- and 600-cell.  

We would like to thank Derek Smith and David Richter for useful discussions about the Coxeter/Weyl group $W$ for the $H_4$ root system.  The third author especially thanks   Prof. Thomas Brylawski for teaching him about matroids and the beauty of symmetry groups.
\section{Preliminaries}\label{S:def}
We assume some basic familiarity with matroids and root systems.  We refer the reader to the first chapter of  \cite{ox} for an introduction to matroids and \cite{gb,h} for much more on root systems.  The study of root systems is very important for Lie algebras, and the term `root' can be traced to characteristic roots of certain Lie operators.  For our purposes, the collection of roots forms a matroid, and the  Coxeter/Weyl  group of the root system is closely related to the automorphism group of that matroid.

%\begin{Def} Let $E$ be  a finite set and $\mathcal{I}$ a family of subsets of $E.$   Then a {\it matroid} is a pair $M=(E,\mathcal{I})$ satisfying:
%\begin{enumerate}
%\item $\emptyset \in \mathcal{I}  $;
%\item If $J \subseteq I$ and $I \in \mathcal{I}$, then $J \in \mathcal{I}$;
%\item If $I, J \in \mathcal{I}$ with $|I|<|J|$, then $I \cup \{x\} \in   \mathcal{I}$ for some $x \in J-I.$
%\end{enumerate}
%
%
%\end{Def}
%
%The family $\mathcal{I}$ is called the {\it independent} sets of the matroid.  In particular, when $E$ is a finite set of vectors over a field, $M=(E,\mathcal{I})$ forms a matroid, where $\mathcal{I}$ is the family of  subsets of $E$ that are linearly independent over the field.  
%
%\begin{Def} For a vector $\bf{v} \in \mathbb{R}^n$, let $r_{\bf{v}}$ denote the reflection in the hyperplane normal to $\bf{v}$.  Then a finite set   $S$ of vectors in $\mathbb{R}^n$ is a {\it root system} if it satisfies:
%\begin{enumerate}
%\item $S \cap \mathbb{R}\bf{v}=\{\bf{v},-\bf{v}\} $;
%\item For each vector $\bf{v}\in S$, we have $r_{\bf{v}}(S)=S$.
%\end{enumerate}
%\end{Def}

The root system $H_4$ has an interpretation via two dual 4-dimensional regular polytopes, the 120-cell and the 600-cell.   The 120-cell is composed of 120 dodecahedra and the 600-cell is composed of 600 tetrahedra.  Each vertex of the 120-cell is incident to precisely 3 dodecahedra and each vertex of the 600-cell meets 5 tetrahedra, justifying the intuitive notion that the 120-cell is a 4-dimensional analogue of the dodecahedron while the 600-cell is a 4-dimensional version of the icosahedron.

As dual polytopes, the 120-cell and the 600-cell have the same set of hyperplane reflections and symmetry groups.  Then the connection between these two dual solids and the root system $H_4$ is direct:  The roots are precisely the normal vectors of all the reflecting hyperplanes that preserve the 120-cell (or, dually, the 600-cell).  A copy of this  root system also appears as the collection of 120 vertices of the 600-cell (where the 600-cell is positioned with the origin at its center and we identify a vertex with  the vector from the origin to that vertex).  Extensive information about these polytopes appears in Table 5 of the appendix of \cite{cox}.

\begin{Def}  The matroid $M(H_4)$ is defined to be the linear dependence matroid on the set  of 60 column vectors of the matrix $H$ over $\Rats[\tau]$, where $\tau=\frac{1+\sqrt{5}}{2}$ satisfies $\tau^2=\tau+1$.

$$H=\left[\begin{array}{cccccccccccc}
1 & 0 & 0 & 0 & 1 & 1 & 1 & 1 & 1 & 1 &1&1 \\
0 & 1 & 0 & 0 & 1 & 1 & 1 & 1 & -1 & -1&-1&-1 \\
0 & 0 & 1 & 0 & 1 & 1 & -1 & -1 & 1 & 1&-1&-1 \\
0 & 0 & 0 & 1 & 1 & -1 & 1 & -1 & 1 & -1&1&-1
\end{array}\right .\dots$$

$$\left  .\begin{array}{cccccccccccc}
0&0&0&0&0&0&0&0&0&0&0&0\\
 \tau &  \tau & \tau & \tau &\tau^2 & \tau^2 & \tau^2 &\tau^2 &1 &1&1 &1  \\
 \tau^2 & \tau^2 &-\tau^2 &-\tau^2 &  1 &1 &- 1&-1&\tau&\tau &-\tau &-\tau   \\
1 &-1&1&-1 &\tau&-\tau&\tau&-\tau& \tau^2&-\tau^2 &\tau^2 &- \tau^2 
\end{array}\right.\dots$$

$$\left . \begin{array}{cccccccccccc}
\tau & \tau & \tau & \tau & \tau^2 & \tau^2 &\tau^2&\tau^2 &1 &1 & 1 &1  \\
 0&  0 & 0 & 0 &0 & 0 & 0 &0 &0 &0 &0 &0  \\
1 & 1 &-1 &-1 & \tau& \tau &- \tau&-\tau &\tau^2&\tau^2 &-\tau^2  &-\tau^2    \\
\tau^2 &-\tau^2&\tau^2&-\tau^2 &1 &-1 &1 &-1 & \tau&-\tau&\tau&-\tau 
\end{array}\right.\dots$$

$$\left  . \begin{array}{cccccccccccc}
\tau& \tau  & \tau  & \tau  &\tau^2  & \tau^2  & \tau^2  & \tau^2  &1  & 1  & 1&1 \\ 
 \tau^2 & \tau^2 & -\tau^2 &-\tau^2 &1 &1&-1 &-1 & \tau&\tau&-\tau &-\tau \\ 
0& 0 &0 &0 & 0 &0 &0 &0 &0&0&0 &0   \\
1 &-1&1&-1 &\tau&-\tau&\tau&-\tau  &\tau^2&-\tau^2&\tau^2&-\tau^2 
\end{array}\right.\dots$$

$$\left .\begin{array}{cccccccccccc} 
\tau & \tau &\tau & \tau &\tau^2 & \tau^2 & \tau^2 &\tau^2 &1 &1 &1 &1  \\
1 & 1 &-1 &-1 & \tau& \tau &-\tau&-\tau &\tau^2&\tau^2 &-\tau^2  &-\tau^2 \\ 
\tau^2&-\tau^2&\tau^2&-\tau^2&1&-1&1&-1&\tau&-\tau&\tau&-\tau    \\
0&0&0&0&0&0&0&0&0&0&0&0 
\end{array}\right]$$

\end{Def}

The full root system $H_4$ consists of these 60 column vectors together with their 60 negatives.  Note that replacing any column vector by its negative does not change the matroid.  See Sec. 8.7 of \cite{cox} for more on the derivation of these coordinates.

Since $r(M(H_4))=4$, we can represent the matroid with an affine picture in $\mathbb{R}^3$.  We find affine coordinates  in $\mathbb{R}^3$ as follows:  First, find a non-singular linear transformation of the column vectors of $H_4$ that maps each vector to an ordered 4-tuple in which the first entry is non-zero, then project onto the plane $x_1=1$ and plot the remaining ordered triples.  We note that choosing different transformations gives rise to different projections;  choosing the `best' such projection is subjective.  In Figure \ref{F:alllines}, we give a projection of one such representation.

%
%If $S \subseteq E$, define the {\it rank} of $S$ by $r(S)=\may_{I \in \mathcal{I}}\{ |I| : I \subseteq S\}.$  A set $F \subseteq E$ is  {\it closed} or a {\it flat} of the matroid if $r(F\cup \{x\})=r(F)+1$ for all $x \in E-F.$  A $k$-point {\it line} $l$ is a rank 2 flat of cardinality $k.$

\section{The flats and orthoframes of $M(H_4)$} \label{S:flats}
We describe the rank-4 matroid $M(H_4)$ by determining the number of flats of each kind and the flat incidence structure.  This incidence structure will also be important for determining the automorphisms of the matroid.  We use  lower case letters to label the points of the matroid and upper case letters for flats of rank 2 or 3.

\subsection{Flats} Every line in $M(H_4)$ has 2, 3 or 5 points, and there are 4 different isomorphism classes of planes (rank-3 flats) in $M(H_4)$.  The planes are shown in Figure~\ref{F:3planes}.  This fact can be proven by a direct computation using the column dependences of the matrix $H$.

\begin{figure}[htbp]
\begin{center}
\includegraphics[width=4in]{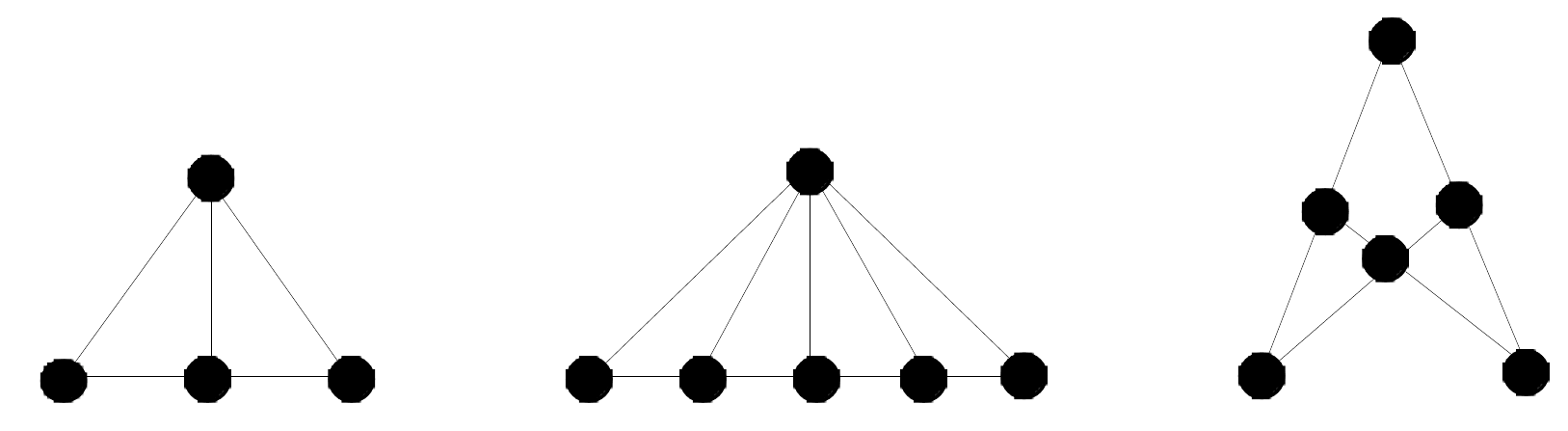}

$\Pi_3$\hspace{1.4in} $\Pi_5$\hspace{1.2in} $\Pi_6$

\bigskip

\includegraphics[width=6in]{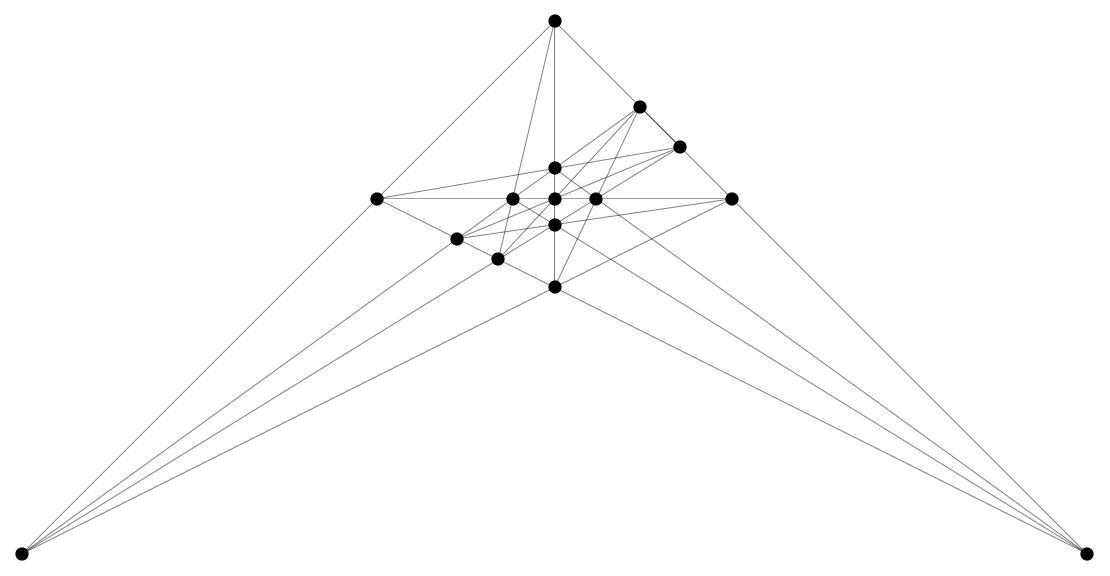}

$\Pi_{15}$
\caption{The four planes that appear in $M(H_4)$.}
\label{F:3planes}
\end{center}
\end{figure}

\begin{Lem}\label{L:flatcount}  Flat counts:  In Table \ref{T:flatcount}, we list the number of flats of rank 1, 2 and 3 in the matroid $M(H_4)$.
\end{Lem}
\begin{table}[htdp]
\caption{The number of flats of each kind in the matroid.}
\begin{center}
\begin{tabular}{|c||c||c|c|c||c|c|c|c|} \hline
Rank & Rank 1& \multicolumn{3}{c||}{Rank 2}  & \multicolumn{4}{c|}{Rank 3}  \\ \hline
Flat& Points & 2-pt lines & 3-pt lines & 5-pt lines  &
 $\Pi_3$ & $\Pi_5$ & $\Pi_6$ & $\Pi_{15}$  \\ \hline 
No. & 60 & 450 & 200 & 72  &
600 & 360 & 300 & 60 \\ \hline
\end{tabular}
\end{center}
\label{T:flatcount}
\end{table}

Diagrams of $M(H_4)$ that emphasize the 3-point lines and 5-point lines appear in Figure \ref{F:35lines}.

\begin{figure}[htbp]
\begin{center}
\includegraphics[width=2.45in]{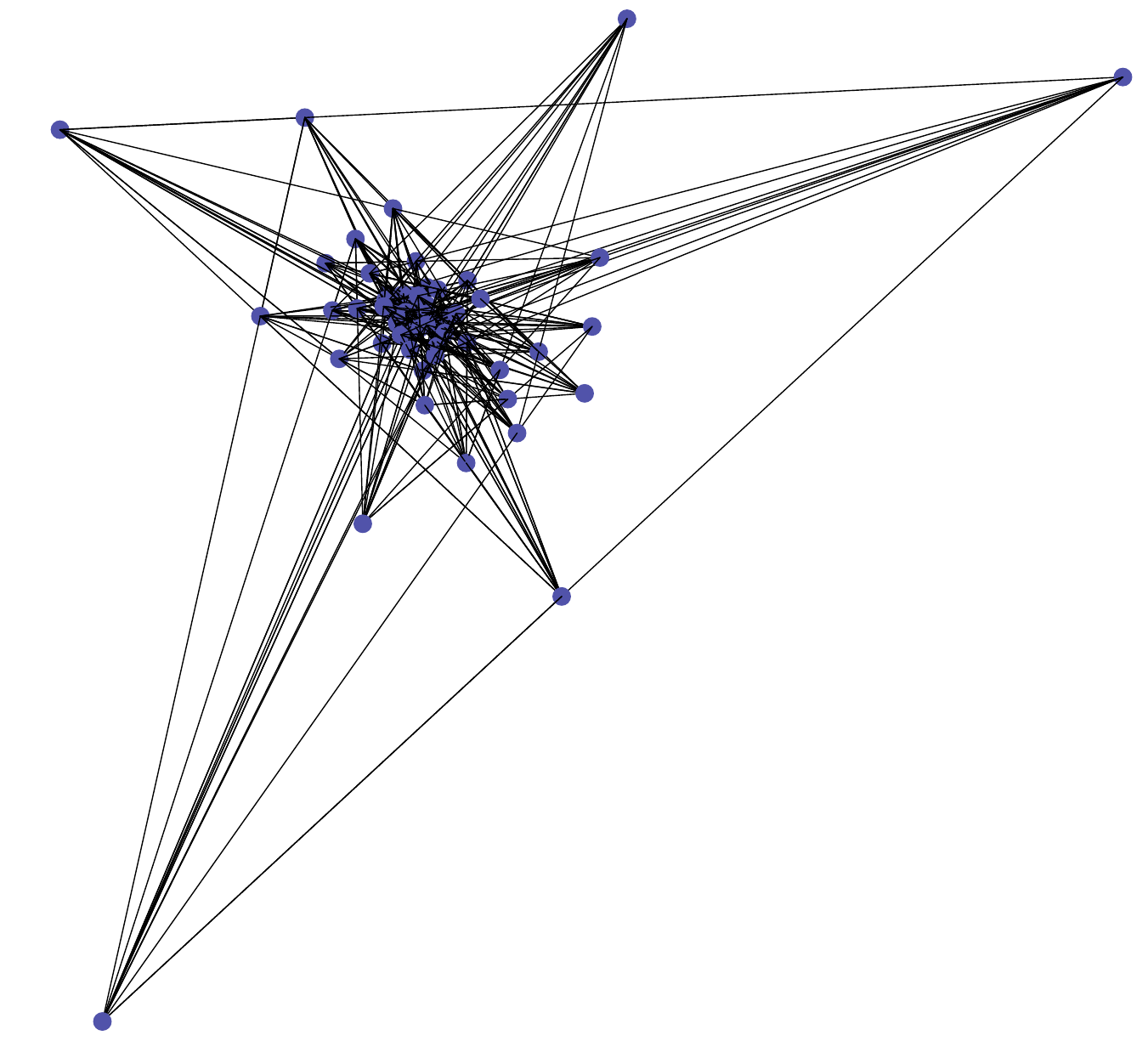}
\includegraphics[width=2.45in]{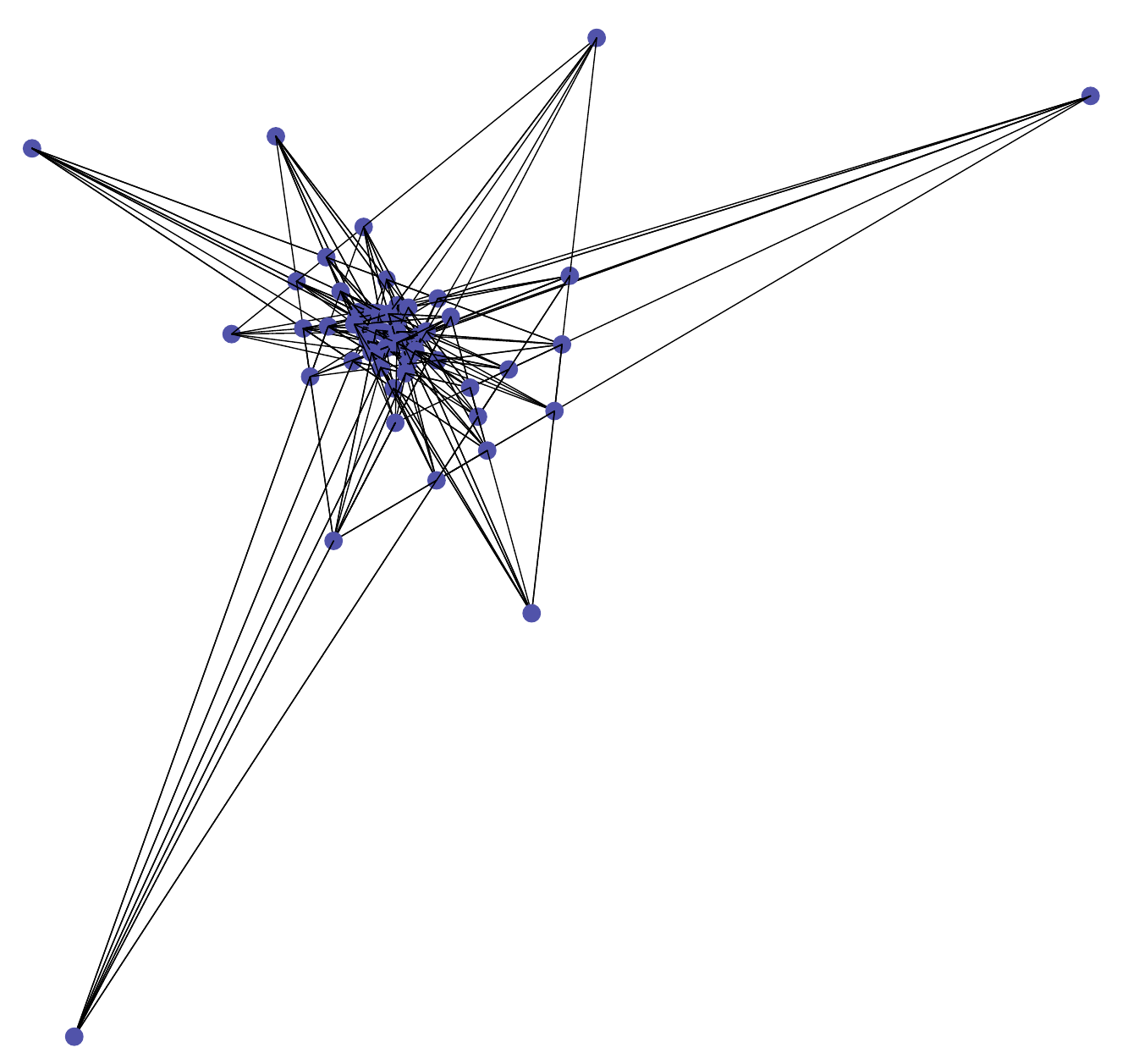}
\caption{The 3-point lines (left) and 5-point lines (right) of $M(H_4)$.}
\label{F:35lines}
\end{center}
\end{figure}

\begin{Lem}\label{L:flatinc} 
Flat incidence:  In Table \ref{T:flatincidence}, we list the number of flats of a certain kind that contain a given flat of lower rank.  
\end{Lem}
\begin{table}[htdp]
\begin{center}
\begin{tabular}{|c||c|c|c||c|c|c|c|} \hline
 & \multicolumn{3}{c||}{Rank 2}  & \multicolumn{4}{c|}{Rank 3}  \\ \hline
 & 2-pt lines & 3-pt lines & 5-pt lines  &
$\Pi_3$ & $\Pi_5$ & $\Pi_6$ & $\Pi_{15}$  \\ \hline 
A point is in & 15 & 10 & 6  & $10^{(a)}$& $6^{(a)}$ & 30 & 15 \\ \hline
A 2-pt line is in& 1 & - & - &4 &4& 2 & 2 \\ \hline
A 3-pt line is in & -&1&-  &3 &- & 6 & 3 \\ \hline
A 5-pt line is in& -& -&1 &- &5 & - & 5 \\ \hline
\end{tabular}
\end{center}

\smallskip
\caption{The number of flats of one kind that contain a given flat of another kind. (a)  The point is the apex of the $\Pi_3$ or $\Pi_5$.}
\label{T:flatincidence}
\end{table}

Both lemmas can be verified by  computer calculations, but we give an example of how the various counts are interrelated. Assuming the point-flat incidence counts for 3-point lines and $\Pi_{15}$ planes, we will count  the number of $\Pi_6$ planes; other counts may be obtained with similar arguments.  

For a given point $x \in M(H_4)$, there are ten 3-point lines through $x$, giving $45$ pairs of 3-point lines containing $x$.  Now $x$ is in 15 $\Pi_{15}$ planes, and each of these planes is completely determined by the pair of 3-point lines containing $x$.  Each of the remaining 30 pairs of 3-point lines containing $x$ uniquely determine a $\Pi_6$ containing $x$.  Thus, there are 30 $\Pi_6$ planes containing a given point.

To get the total number of $\Pi_6$ planes, consider the point-$\Pi_6$ incidence.  Each point is in 30 $\Pi_6$ planes, and each $\Pi_6$ contains 6 points.  Thus, the total number of $\Pi_6$ planes is $300$.

The  flats of a matroid satisfy the {\it flat covering property}:

\begin{quotation}
  If  $F$ is a  flat in a matroid $M$, then $\{F'-F \mid F' \mbox{ is a flat that covers } F\}$ partitions $E-F$.
\end{quotation}

We illustrate this partitioning property for  $M(H_4)$:

\begin{description}
\item[Point/line incidence]   From Table~\ref{T:flatincidence}, we know a given point $x$ is covered by precisely 15 2-point lines, ten 3-point lines and six 5-point lines.  Then it is easy to see this pencil of lines contains precisely 59 points (not counting $x$), partitioning $E-x$, as required.
\item [Line/plane incidence]  We consider the three kinds of lines in $M(H_4)$.
\begin{itemize}
\item 2-point lines:  Each 2-point line $L$ is covered by four $\Pi_3$'s, four $\Pi_5$'s, two $\Pi_6$'s and two $\Pi_{15}$'s.  Each $\Pi_3$ that covers $L$ contains two points not on $L$.  Similarly, each $\Pi_5$ covering  $L$ has four more points, each such $\Pi_6$ also has four more points, and each such  $\Pi_{15}$ has 13 points.  This gives us the required partition of the remaining 58 points.
\item 3-point lines:  Each 3-point line $L$ is in 3 $\Pi_3$'s, 6 $\Pi_6$'s and 3 $\Pi_{15}$'s.  As above, counting the points in these covering planes gives a total of 57 points partitioned by these planes.
\item 5-point lines:   If $L$ is a 5-point line, then only two kinds of planes contain $L$: the 5 $\Pi_5$'s and the 5 $\Pi_{15}$'s.  These 10 planes contain 55 points (excluding the points on $L$), again giving us the required partition of $E-L$.
\end{itemize}

\end{description}

As an application of the incidence data given above, we prove the following.

\begin{Prop}\label{P:planeintersect} Every pair of $\Pi_{15}$ planes intersect.
\end{Prop}

\begin{proof}
Let $P$ be a 15-point plane and let $L_5$ be a 5-point line contained in $P$.  Since every 5-point line is contained in precisely five $\Pi_{15}$'s, there are four $\Pi_{15}$'s that meet $P$ along the line $L_5$.  Since $P$ contains six 5-point lines, this gives a total of 24 $\Pi_{15}$'s that meet our given plane $P$ in a 5-point line.

We repeat this argument for 3-point lines:  Each of the ten 3-point lines in $P$ is contained in two more $\Pi_{15}$'s, accounting for another 20 $\Pi_{15}$'s meeting $P$.

Finally, each two-point line is in two $\Pi_{15}$'s, but there are 15 2-point lines in $P$.  This gives another 15 $\Pi_{15}$'s planes that meet $P$ in a 2-point line.  But this now accounts for 59 $\Pi_{15}$'s, all of which meet $P$ in either a 2, 3 or 5-point line.  Thus, every pair of $\Pi_{15}$'s meet.

  \end{proof}
  
  In fact, all of these intersections are {\it modular}:  $r(P_1 \cap P_2)=2$ for all pairs of 15-point planes $P_1$ and $P_2$.  We also remark the 15-point planes are isomorphic (as matroids) to $M(H_3)$ - the matroid associated to the root system $H_3$ (see \cite{eg}).  We will need this connection in Section~\ref{S:aut}.

\subsection{Orthoframes}
Of special interest is the interesting symmetry between points and $\Pi_{15}$ planes:  There are 60 points and 60 $\Pi_{15}$'s, where each point is in 15 $\Pi_{15}$'s and each $\Pi_{15}$ has 15 points.  The easiest way to understand this symmetry is through {\it orthoframes}.

\begin{Def}\label{D:ortho}
A basis $B$ for $M(H_4)$ is an {\it orthoframe} if each pair of points in $B$ forms a 2-point line in the matroid.
\end{Def}
For instance, the basis formed by the first 4 columns of the matrix $H$ is an  orthoframe.   In general, these bases correspond to column vectors in $H$ that are pairwise orthogonal.  Two more orthoframes are:

$$\left[\begin{array}{cccc}
0& 1 & \tau & \tau^2 \\
1 & 0 & \tau^2 & -\tau  \\
-\tau & \tau^2 & 0 & -1 \\
-\tau^2 & -\tau& 1 & 0
\end{array}\right] \hspace{.5in}
\left[\begin{array}{cccc}
1& 0 & 1 & \tau^2 \\
1 & \tau & \tau & -\tau  \\
1 & -\tau^2 & 0 & -1 \\
1 & 1& -\tau^2 & 0
\end{array}\right]
$$

Orthoframes are important to us for two reasons:   matroid automorphisms give  group actions on the set of orthoframes, and orthoframes have an immediate geometric interpretation  in the root system $H_4$.  

We state (without proof) several useful facts we will need about orthoframes.  The proofs are routine, and follow in a similar way the incidence counts of Lemmas~\ref{L:flatcount} and \ref{L:flatinc}.

\begin{Prop}\label{P:orthodata}
\begin{enumerate}
\item $B$ is an orthoframe if and only if  the four column vectors corresponding to the points of $B$   are pariwise orthogonal.
\item There are 75 orthoframes.
\item Each point is in 5 orthoframes, and each 2-point line is in exactly one orthoframe.  
\item If $O_1, O_2, \dots, O_5$ are the 5 orthoframes that contain a given point $x$, then $\displaystyle{\bigcup_{i=1}^5 O_i -x}$ is a 15-point plane.
\end{enumerate}

\end{Prop}

Part (4) of this proposition allows us to define a bijection between the points of the matroid and the 15-point planes: Given a point $x$, let $O_1, O_2, \dots, O_5$ be the five orthoframes that contain $x$.  Then define $P_x:=\displaystyle{\bigcup_{i=1}^5 O_i -x}$. Conversely, a given 15-point plane can be partitioned into five partial orthoframes (this partition is visible in the picture of a $\Pi_{15}$ in Fig.~\ref{F:3planes} -- see also Sec. 2.1 of \cite{eg}).  Then a 15-point plane $P_x$ uniquely determines a point $x$ that ``completes'' each of these orthoframes.

We will use this correspondence frequently; we introduce some terminology suggestive of the relationship between the column vectors corresponding to the point and the plane.
\begin{Def}\label{D:names}
Suppose the point $x$ corresponds to the 15-point plane $P_x$ as above.  Then we say the point $x$ is the {\it orthopoint} of the plane $P_x$ and the 15-point plane $P_x$ is the {\it orthoplane} of the point $x$.
\end{Def}

We can also use the orthoframes to uniquely reconstruct the matroid $M(H_4)$.   

\begin{Prop}\label{P:orthorecon}  The collection of 75 orthoframes completely determines all the flats of the matroid $M(H_4)$.

\end{Prop}
\begin{proof}   We show how the orthoframe data allows us to reconstruct all the flats.
\begin{itemize}
\item {\bf $\Pi_{15}$ planes:}  The union of the orthoframes containing a given point $x$ form the 15-point orthoplane $P_x$ (where the common point is removed), so we can construct all the $\Pi_{15}$'s this way.
\item {\bf Lines:}  Since each 2-point line is in a unique orthoframe, we simply list the six 2-point lines contained in each of the 75 orthoframes, giving us the 450 2-point lines.   By the proof of Prop.  \ref{P:planeintersect}, every 3-point line and every 5-point line occurs as the intersection of some pair of $\Pi_{15}$'s.  This allows us to reconstruct all rank-2 flats. 
\item {\bf $\Pi_3$ and $\Pi_5$ planes:} For the  trivial planes $\Pi_3$ and $\Pi_5$, each such plane arises as the union of  a 3 or 5-point line in a $\Pi_{15}$ with the plane's orthopoint as the apex of the $\Pi_3$ or $\Pi_5$.
\item {\bf $\Pi_6$ planes:} The remaining non-trivial flats are the 300  $\Pi_6$ planes.  We consider all pairs of intersecting 3-point lines.  Each intersecting pair determines either a $\Pi_{15}$ or a $\Pi_6$.  We know all the 15-point planes at this point, so we can determine all pairs giving a $\Pi_6$.  To reconstruct each $\Pi_6$ from this information, note that each $\Pi_6$ contains four 3-point lines, every pair of which intersect.  This allows us to uniquely determine each $\Pi_6$ from the collection of 3-point lines.
\end{itemize}

\end{proof}

Compared with bases, orthoframes provide a much more efficient way to describe the matroid.  While there are 75 orthoframes, a computer search gives 398,475 bases; a random subset of four columns has approx 81.7\% chance of being a basis.  

%The point-orthoplane correspondence also gives us an easy way to describe all of the $\Pi_3$ and $\Pi_5$ planes.  Recall that there are ten 3-point lines and six 5-point lines in each $\Pi_{15}$.  Adding the orthopoint $x$ to each of these lines gives us ten $\Pi_3$'s and six $\Pi_5$'s, all of which use the point $x$ as their apex.  Since there are 60 $\Pi_{15}$'s, we get 600 $\Pi_3$'s and 360 $\Pi_5$'s in this way.  But that accounts for all the $\Pi_3$'s and $\Pi_5$'s given in Lemma~\ref{L:flatcount}.  (Note, since each 3-point line is in three $\Pi_{15}$'s and each 5-point line is in five $\Pi_{15}$'s,  this procedure uses each 3-point line three times, and each 5-point line five times.)

We conclude this section by noting  an algebraic explanation for the point-orthoplane correspondence.  Each point corresponds to an ordered 4-tuple $[a,b,c,d]$, and  each $\Pi_{15}$ corresponds to the solution set of a linear equation.  The connection between the coordinates of the point $z$ and the corresponding linear equation the associated orthoplane $P_z$ satisfies is simple.

\begin{Prop}\label{P:planeeqns}
Let $z$ be a point with corresponding orthoplane $P_z$, and suppose $z$ corresponds to the ordered 4-tuple $[a,b,c,d]$.  Then $P_z$ is defined  by the linear equation $ax_1+bx_2+cx_3+dx_4=0$.
\end{Prop}
\begin{proof}
Let $z$ be a point and let $O_1, O_2, \dots, O_5$ be the five orthoframes containing $z$.  Then if  $\displaystyle{y \in \bigcup_{i=1}^5 O_i -z =P_z}$, we have the 4-tuples corresponding to the points $y$ and $z$ are orthogonal (by Prop.~\ref{P:orthodata}(1)).  Thus,  if the coordinates for $z$ are $[a,b,c,d]$, we have  $ax_1+bx_2+cx_3+dx_4=0$ for all column vectors $[x_1,x_2,x_3,x_4] \in P_z$.

\end{proof}
As an example of this algebraic connection, let $z$ be the point with coordinates $[\tau^2,0,\tau,-1]$.  Then the equation $\tau^2 x_1+\tau x_3-x_4=0$ is satisfied by $P_z$:

$$\left(
\begin{array}{ccccccccccccccc}
 0 & 1 & 1 & 0 & 0 & 0 & 0 & \tau & 1 & 1 & 1 & \tau & \tau & 1 & 1 \\
 1 & 1 & -1 & \tau^2 & \tau^2 & 1 & 1 & 0 & 0 & \tau & -\tau & 1 & -1 & \tau^2 & -\tau^2 \\
 0 & -1 & -1 & 1 & -1 & \tau & -\tau & -1 & -\tau^2 & 0 & 0 & -\tau^2 & -\tau^2 & -\tau & -\tau \\
 0 & 1 & 1 & \tau & -\tau & \tau^2 & -\tau^2 & \tau^2 & -\tau & \tau^2 & \tau^2 & 0 & 0 & 0 & 0
\end{array}
\right).$$

\section{Automorphisms}\label{S:aut}
We turn to our main topic: the structure of the automorphisms of $M(H_4)$.  
  For a group $G$ acting on a set $X$ with $x \in X$, recall the {\it stabilizer of } $x$ 
$$\stab(x) = \{g \in G \mid g(x)=x\}.$$
\begin{Lem}\label{L:stab}
Let $x$ be a point of $M(H_4)$ and $P_x$ its 15-point orthoplane.  Then $\stab(x)=\stab(P_x)\cong S_5 \times \Ints_2$.
\end{Lem}
\begin{proof}
The point-orthoplane correspondence (Prop.~\ref{P:orthodata}(4) or Prop.~\ref{P:planeeqns}) gives $\stab(x)= \stab(P_x)$.  Note that $P_x  \cong M(H_3)$, the matroid associated with the icosahedral root system.  Then, by  Theorem 3.3 of \cite{eg}, $\Aut(M(H_3)) \cong S_5$, so $S_5$ fixes the plane $P_x$.  ($S_5$ acts on the five rank-3 orthoframes.)  Thus $S_5 \leq \stab(x)$.

We may now suppose $\sigma \in  \stab(P_x)$ where $\sigma$ fixes the orthoplane $P_x$ pointwise.  We will show that $\sigma =I$ or $\sigma$ is the matroid automorphism induced by geometric reflection $r_x$ of the root system through $P_x$.  (Note that reflection fixes $P_x$ pointwise, but also fixes the orthopoint $x$.)

So assume $\sigma(w)=w$ for $w=x$ and for all $w \in P_x$.  Then $\sigma$ fixes (at least) 16 points; we partition the remaining 44 points of the matroid into two classes:

\begin{description}
\item[Class 1]  Let $\{L_1, L_2, \dots, L_{10}\}$ be the pencil of 3-point lines through $x$.  Then $\mathcal{C}_1:=\bigcup L_i - x$ contains 20 points.  We write $L_i=\{x, y_i, z_i\}$ for $1 \leq i \leq 10$, so $\mathcal{C}_1=\{y_1, y_2, \dots, y_{10}, z_1, z_2, \dots, z_{10}\}$.  For a given $i$, we first show $\sigma$ either fixes both $y_i$ and $z_i$ or it swaps them.  

\begin{figure}
\begin{center}
\includegraphics[width=2in]{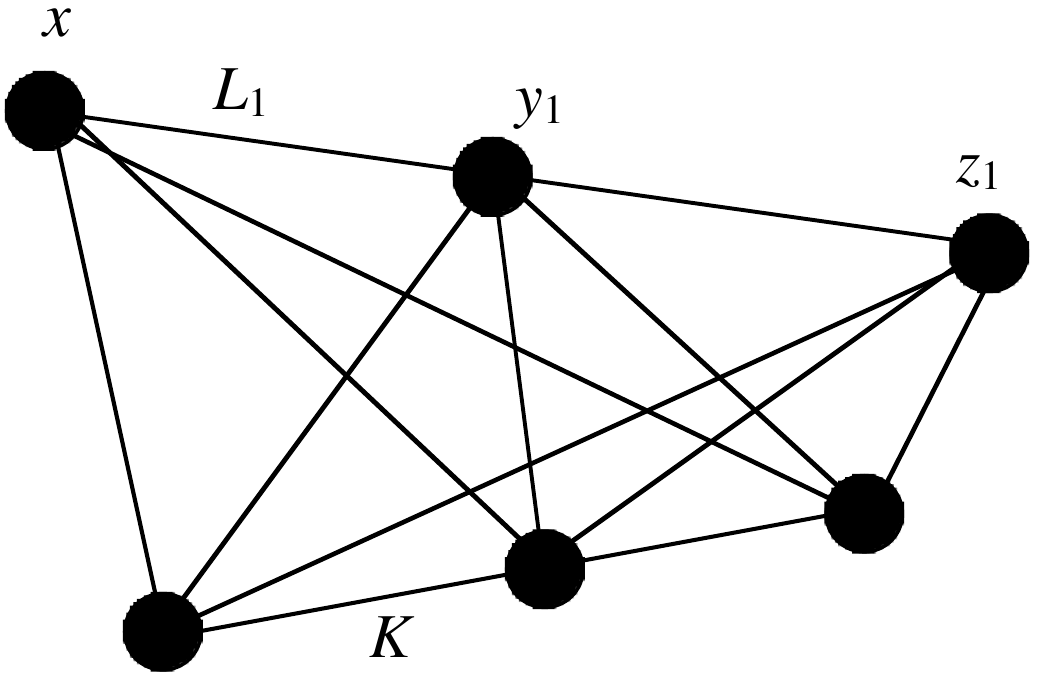}
\caption{The three points $x,  y_1$ and $z_1$ forming the apexes of the three $\Pi_3$'s containing the line $K$ are collinear.  $K$ and $L_1$ are skew, i.e, $r(K \cup L_1)=4$.}
\label{F:3ptlines}
\end{center}
\end{figure}

Consider a 3-point line $K$ in the 15-point plane $P_x$, which we know is fixed pointwise by $\sigma$.  Then $K$ is also fixed pointwise.  The line $K$ is contained in three $\Pi_3$'s (Lemma~\ref{L:flatinc}), with three different apexes, one of which is $x$.  Then it is straightforward to show these three apexes form a 3-point line, so they correspond to one of the lines, say $L_1$, in the pencil through $x$, as in Figure~\ref{F:3ptlines}.  Since matroid automorphisms preserve all $\Pi_3$'s, and since $K$ is fixed, we must have $\sigma(y_1)\in\{y_1,z_1\}$.  

But there are ten 3-point lines in $P_x$, and each of these lines will correspond to one of the $L_j$ in precisely the same way $K$ corresponds to $L_1$.  Thus, we have $\sigma(y_i)\in \{y_i,z_i\}$ for all $i$.

Now  suppose $\sigma(y_1)=y_1$.  We will show that $\sigma(y_i)=y_i$ for all $i$ (and so $\sigma$ is the identity on $\mathcal{C}_1$).  Now every pair of lines $L_i, L_j$ determines either a 6-point plane $\Pi_6$ or a 15-point plane $\Pi_{15}$.  Our incidence counts from Lemma~\ref{L:flatinc} can be used to show that, for a given $i$, precisely 6 lines $L_j$ can be paired with $L_i$ to generate a $\Pi_6$, and the remaining three lines will generate  $\Pi_{15}$'s when paired with $L_i$.  We concentrate on the $\Pi_6$'s.

\begin{figure}
\begin{center}
\includegraphics[width=3in]{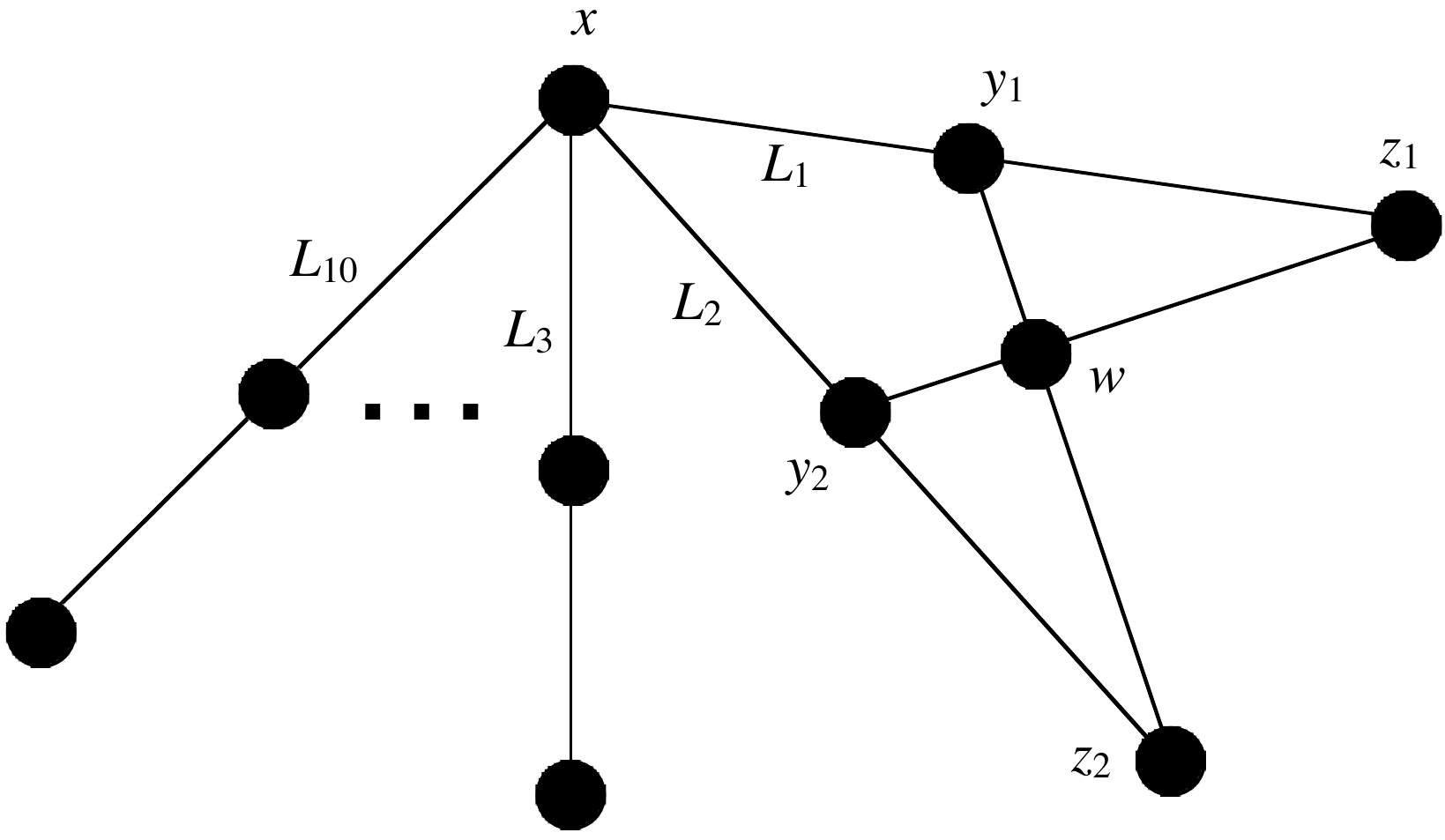}
\caption{The pencil of 3-point lines through $x$.  $L_1$ and $L_2$ generate a $\Pi_6$.}
\label{F:6ptplanes}
\end{center}
\end{figure}

Suppose $L_1$ and $L_2$ determine a $\Pi_6$, where $w$ is the unique point of the $\Pi_6$ not on $L_1$ or $L_2$, as in Figure~\ref{F:6ptplanes}.   Since $\sigma(x)=x$ and $\sigma(y_1)=y_1$, we know $\sigma(z_1)=z_1$.  Thus, if $\sigma$ swaps $y_2$ and $z_2$, then the 3-point line $\{y_1, w, z_2\}$ is mapped to the independent set $\{y_1, w, y_2\}$, which is impossible for a matroid automorphism.  Thus, $\sigma$ fixes $y_2$ and $z_2$.

To show that $\sigma$ fixes all $y_i$ and $z_i$, construct a graph $\Gamma$ as follows:  The 10 vertices are labeled by the lines $L_i$, with an edge between $L_i$ and $L_j$ if and only if  these two lines determine a $\Pi_6$.  Then $\Gamma$ is a regular graph on 10 vertices with every vertex having degree 6, so $\Gamma$ is connected.  Thus we can find a path from $L_1$ to any line $L_j$, and it is clear that each edge of the path forces $\sigma$ to fix the points on the corresponding line.  Thus, $\sigma$ fixes each point in $\mathcal{C}_1$.  (Incidentally, we note the point $w$ is on the fixed 15-point plane $P_x$.  By choosing different pairs of lines in the pencil, we can locate all 15 points of $P_x$ in this way.)

Finally, if $\sigma$ swaps any pair $y_i, z_i$, then  $\sigma$ swaps {\it all} pairs, by a similar argument.  Then $\sigma$ corresponds to the reflection $r_x$ through $P_x$.

\smallskip

\item[Class 2]  Let $\{M_1, M_2, \dots, M_{6}\}$ be the pencil of six 5-point lines through $x$ (again, from Lemma~\ref{L:flatinc}).  Then $\mathcal{C}_2:=\bigcup M_i - x$ contains 24 points.  As we did for $\mathcal{C}_1$, we show that these 24 points are either swapped in 12 transpositions (when $\sigma$ corresponds to reflection) or are all fixed pointwise (when $\sigma=I$).

As before, fix a 5-point line $K$ in the fixed plane $P_x$ and consider the five points in the matroid that form the apexes of $\Pi_5$ planes which use $K$.  Then it is again straightforward to show that these five apexes form a 5-point line, so they correspond to one of the $M_i$.  (This is completely analogous to the situation with $\Pi_3$'s that contain a fixed 3-point line, as in Figure~\ref{F:3ptlines}.)  Since matroid automorphisms preserve $\Pi_5$'s, each line $M_i$ in the pencil must be fixed.

We need to show that $\mathcal{C}_2$ is fixed by $\sigma$ when $\sigma$ fixes $\mathcal{C}_1$ pointwise.  Now the 15 pairs of lines in the pencil $\{M_1, M_2, \dots, M_{6}\}$ generate the 15 $\Pi_{15}$ planes containing $x$.   Thus, if $\sigma$ fixes $\mathcal{C}_1$ pointwise, it fixes two intersecting 3-point lines in each of these $\Pi_{15}$'s, since the $\Pi_{15}$'s containing $x$ are also generated by 15 pairs of lines from $\{L_1, L_2, \dots, L_{10}\}$.  Thus, in each $\Pi_{15}$ that contains $x$, we have a pair of intersecting 3-point lines that are fixed pointwise, and a pair of intersecting 5-point lines that are also fixed (not necessarily  pointwise).  

But the only automorphism of a 15-point plane with this cycle structure on its 3- and 5-point lines is the identity -- this follows from the last two columns of Table 1 of \cite{eg}.  Thus, $\sigma$ fixes $\mathcal{C}_2$ pointwise.

If $\sigma$ swaps each pair $(y_i,z_i)$ in $\mathcal{C}_1$, then we obtain reflection again, and the 24 points in $\mathcal{C}_2$ are all moved in 12 transpositions, corresponding to the reflection $r_x$ through the plane $P_x$.
\end{description}

Thus, every $\sigma \in \stab(x)$ can be decomposed as an automorphism of the plane $P_x$ followed or not by reflection through that plane.  These two operations commute, so we have $\stab(x) \cong S_5 \times \Ints_2$.

\end{proof}

Recall there are seven different equivalence classes of flats: 2, 3 and 5-point lines, and 4 different classes of planes.
\begin{Lem}\label{L:trans}
$\Aut(M(H_4))$ acts transitively on each equivalence class of flats of the matroid.
\end{Lem}
\begin{proof}
The proof makes use of the fact that the Coxeter/Weyl group acts transitively on the roots of $H_4$ (see \cite{cox}).  Since every geometric symmetry of the root system gives rise to a matroid automorphism, we immediately get $\Aut(M(H_4))$ acts transitively on the points of the matroid.   The point-orthoplane correspondence then gives us a transitive action on the $\Pi_{15}$'s.

We now consider the remaining flat classes.

\begin{description}
\item[Rank 2 flats] Let $\mathcal{L}_k$ be the class of all $k$-point lines, for $k=2,3$ and 5, and let $L_1$ and $L_2$ be two $k$-point lines.  If $L_1$ and $L_2$ are both in the same $\Pi_{15}$, then we use the fact (see \cite{eg}) that $\Aut(M(H_3))$ acts transitively on lines to get an automorphism $\sigma$ mapping $L_1$ to $L_2$. 

 If $L_1$ and $L_2$ are not contained in any $\Pi_{15}$, then either $r(L_1\cup L_2)=4$, i.e., the lines $L_1$ and $L_2$ are skew, or $L_1$ and $L_2$ are 3-point lines in a $\Pi_6$.  In the former case, find two 15-point planes $P_1$ and $P_2$ with $L_1 \subseteq P_1$ and $L_2 \subseteq P_2$.  Now use the transitivity on 15-point planes to map $P_1$ to $P_2$, and then use transitivity on lines within $P_2$ to map the image of $L_1$ to $L_2$.
 
 If $L_1=\{a,b,c\}$ and $L_2=\{a,d,e\}$ are intersecting 3-point lines in a $\Pi_6$, then use transitivity on points to map $b$ to $d$.  This must carry $L_1$ to $L_2$.
 
 \smallskip
 \item[Rank 3 flats]  We already have $\Aut(M(H_4))$ is transitive on 15-point planes.  It is also clear that transitivity of 3- and 5-point lines gives us transitivity on $\Pi_3$ and $\Pi_5$ planes.  It remains to prove transitivity for $\Pi_6$ planes.
 
 Let $G_1$ and $G_2$ be two $\Pi_6$ planes, and let $L_1$ and $L_2$ be 2-point lines with $L_i \subseteq G_i$ ($i=1$ or 2).  Then transitivity on 2-point lines allows us to map $L_1 \mapsto L_2$.  So we can assume $G_1$ and $G_2$ share the 2-point line $xy$, as in Figure~\ref{F:6ptplanestrans}.  By Lemma~\ref{L:flatinc}, $G_1$ and $G_2$ are the only two $\Pi_6$'s that contain $xy$.  Then there are two 15-point planes that also contain the 2-point line $xy$; call these two planes $P_1$ and $P_2$.  Then reflecting through either $P_1$ or $P_2$ will map $G_1$ to $G_2$, since reflection must send a $\Pi_6$ to a $\Pi_6$, reflections move 44 points, and $G_1$ and $G_2$ are the only two $\Pi_6$'s containing $xy$.
 
\end{description}

\begin{figure}
\begin{center}
\includegraphics[width=4in]{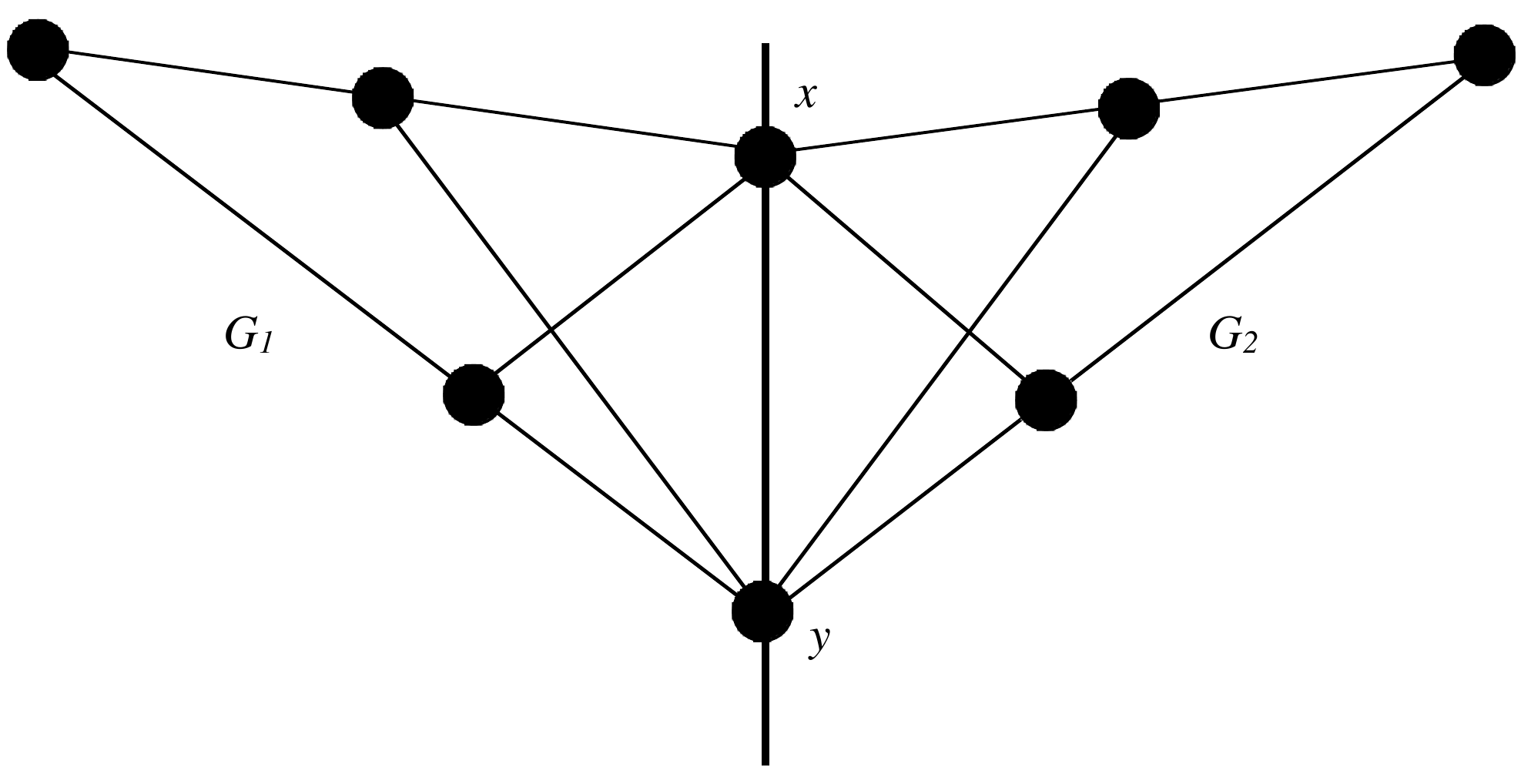}
\caption{Two $\Pi_6$ planes share the 2-point line $xy$. $r(G_1 \cup G_2)=4$.}
\label{F:6ptplanestrans}
\end{center}
\end{figure}

\end{proof}

As an example of how transitivity on $\Pi_6$ planes works, consider the matrices $A$ and $B$ below.  The columns of $A$ satisfy the equation $x_3=x_4$, and the columns of $B$ satisfy $x_1=x_2$.  Note that the corresponding 6-point planes  have two points in common - the 2-point line $ef$.
$$
A=\kbordermatrix{&a&b&c&d&e&f\\
& 1 & 0 & 1 & 1 & 1 & 1 \\
& 0 & 1 & -1&  -1 & 1 & 1\\
& 0 & 0 & 1 & -1 & -1& 1\\
& 0 & 0 & 1 & -1 & -1& 1}
\hskip.25in
B=\kbordermatrix{&a'&b'&c'&d'&e&f\\
& 0 & 0 & 1 & 1 & 1 & 1 \\
& 0 & 0 & 1&  1 & 1 & 1\\
& 1 & 0 & 1 & -1 & -1& 1\\
& 0 & 1 & -1 & 1 & -1& 1}
$$

To find a matroid automorphism that maps $G_1$ to $G_2$, we let $x=[1,-1,1,-1]$ and $y=[1,-1,-1,1]$.  Then $\{e,f,x,y\}$ is an orthoframe, i.e., $e,f \in P_x$ and $e, f \in P_y$.  Reflection through the plane $P_x$ is accomplished by $\displaystyle{v\mapsto v-\frac{2v\cdot x}{x\cdot x}x}$.  This maps $a \mapsto d', b \mapsto c', c \mapsto b', d \mapsto a'$.  The reader can check reflection through $P_y$ maps $a \mapsto c', b \mapsto d', c \mapsto a', d \mapsto b'$.  In either case, we have a map interchanging $G_1$ and $G_2$.

Alternatively,  we can map one plane to the other by performing two row swaps on the matrix $H$: $(13)(24)$. This is an {\it even} permutation of the rows, and so maps $M(H_4)$ to itself.

It is interesting to note that although $\Aut(M(H_4))$ acts transitively on pairs of intersecting 5-point lines, it does not act transitively on pairs of intersecting 3-point lines.  The latter fall into two equivalence classes, as we have already seen:  A pair of intersecting 3-point lines determines either a $\Pi_6$ or a $\Pi_{15}$.

$\Aut(M(H_4))$ also acts transitively on orthoframes.  We omit the short proof.

\begin{Lem}\label{L:orthotrans}
$\Aut(M(H_4))$ acts transitively on orthoframes.
\end{Lem}
Recall a group $G$ acting on a set $X$ is {\it primitive} if $G$ acts transitively and preserves no non-trivial blocks of $X$.  We now prove the action of $\Aut(M(H_4))$ on the points of the matroid is primitive.

\begin{Thm}\label{T:prim}  The automorphism  group action is primitive on the 60 points of the ground set of $M(H_4)$.
\end{Thm}

\begin{proof}  Suppose $E$ is partitioned into blocks, and suppose $\Delta$ is a block.  Then, for any $\sigma \in \Aut(M(H_4)), \Delta \cap \sigma(\Delta)=\Delta$ or $\emptyset$.  We must prove $|\Delta|=1$ or 60.  

Suppose $x\in \Delta$.  Note for all $\sigma \in \stab(x)$, we must have $ \sigma(\Delta)=\Delta$.  Since $\Aut(M(H_3))\cong S_5 \leq \stab(x)$ acts transitively on the 15 points of  $P_x$, we must have either $P_x \subseteq \Delta$ or $P_x \cap \Delta = \emptyset.$  There are now two cases to consider.

\begin{itemize}
\item If $P_x \subseteq \Delta$, then since $P_x$ meets every other 15-point plane (from Prop.~\ref{P:planeintersect}), we get $\sigma(P_x) \cap P_x \neq \emptyset$ for all $\sigma \in \Aut(M(H_4))$.  Thus, $\Delta \cap \sigma(\Delta) = \Delta$ for all $\sigma \in \Aut(M(H_4))$, i.e., $\sigma(\Delta)=\Delta$ for all $\sigma$.  But this immediately gives $\Delta=E$, i.e., $\Delta$ is the trivial block formed by the entire ground set of the matroid.

\item If $P_x \cap \Delta = \emptyset$, we restrict to stab$(x)$ and consider all the lines that contain $x$.  We know $x$ is in 15 2-point lines, but the 15 points that produce these 2-point lines form $P_x$, so none of these 15 points is in $\Delta$.  

There are ten 3-point lines through $x$, which we denote $\{L_1, L_2, \dots, L_{10}\}$, as in the proof of Lemma~\ref{L:stab}.  From that proof and the fact that the action of  $\Aut(M(H_4))$ is transitive on 3-point lines (Lemma~\ref{L:trans}), we must have either $L_i \subseteq \Delta$ for all $1 \leq i \leq 10$, or $\Delta \cap L_i = \{x\}$ for all $i$.  (Note:  Every $\sigma \in \mbox{ stab}(x)$ maps the pencil of lines through $x$ to itself, so each line contributes the same number of points to $\Delta$, and reflecting through the plane $P_x$ forces us to take 0 or 2 points from each $L_i$, not counting $x$.)  Thus, $\Delta$ contains either 0 points or 20 points from the $L_i$ pencil, not counting $x$.

Using an analogous argument on the pencil of six 5-point lines through $x$, we find each such line must meet $\Delta$ in the same number of points, and that number must be 0, 2 or 4 per line (not counting $x$). This means $\Delta$ contains 0, 12 or 24 points from this pencil, again not counting $x$.

Putting all of  this together gives the $|\Delta|=1, 13, 21, 25, 33$ or $45$.  But $|\Delta|$ must divide 60, since the blocks partition $E$.  Thus $|\Delta|=1$, so $\Delta=\{x\}$.
\end{itemize}

\end{proof}

The next result follows immediately from  Theorem 1.7 of \cite{cam}.
\begin{Cor}\label{C:maxsubgp}
$\stab(x)$ is a maximal subgroup of $\Aut(M(H_4))$.
\end{Cor}

It is worth pointing out that the action of $\Aut(M(H_3))$ on $M(H_3)$  is imprimitive -- the partition into rank-3 orthoframes is a non-trivial partition of the 15 elements of the matroid into 5 blocks.  This corresponds geometrically to permuting the 5 cubes embedded in a dodecahedron.

The root system $H_4$ has a Coxeter/Weyl group of size 14,400.  Coxeter's notation   \cite{cox}  for the group  $[3,3,5]$ suggests its construction as a reflection group.  
$$[3,3,5] = \langle R_1, R_2, R_3, R_4 \mid (R_1R_2)^3=(R_2R_3)^3=(R_3R_4)^5=I \rangle$$
In this presentation, we assume each $R_i$ is a reflection, i.e., $R_i^2=I$, and that $(R_iR_j)^2=I$ for $|i-j|>1$, i.e., reflections $R_i$ and $R_j$ are orthogonal for $|i-j|>1$.

Conway and Smith (Table 4.3 of \cite{cs})  express this group as $\pm [I \times I] \cdot 2$, where $I \cong A_5$ is the chiral (or direct) symmetry group of the icosahedron.  In 4-dimensions, $I \times I$ is best understood as a rotation group via quaternion multiplication.  

\begin{Thm}\label{T:aut}  Let $W$ be the Coxeter/Weyl isometry group for the root system $H_4$, with center $Z$ generated by central inversion $\bf{v} \mapsto -\bf{v}$.
\begin{enumerate}
\item $|\Aut(M(H_4))|=|W|=14,400$.
\item $W/Z$ is an index $2$ subgroup of $\Aut(M(H_4))$.
\end{enumerate}

\end{Thm}
\begin{proof}
\begin{enumerate}
\item From Lemma~\ref{L:stab}, we have $\stab(x)\cong S_5 \times \Ints_2$.  Since the orbit of $x$ is all of $E$ (as the automorphism group is transitive), we have $|\Aut(M(H_4))|=|S_5 \times \Ints_2| \cdot |E|=14,400$.
\item Every isometry of $W$ gives a matroid automorphism, and central inversion in $W$ corresponds to the identity in $\Aut(M(H_4))$.  The result now follows from (1).
\end{enumerate}

\end{proof}

In \cite{dut}, $\Aut(M(H_4))$ is obtained as follows:  First extend  the root system $H_4$ by adding an isomorphic copy $H_4 '$ of $H_4$.  Then $\Aut(M(H_4)) \cong W(H_4 \cup H_4')/Z$, where $Z \cong \Ints_2$ is the subgroup generated by central inversion ($Z$ is the center of $W$).    The $H_4'$ copy is obtained by using the field automorphism $\phi:\Rats[\tau] \to \Rats[\tau]$ given by $\tau \mapsto \bar{\tau}$ on the original root system $H_4$.  (Note that this map must operate on a different set of coordinates than those treated here, since the 24 roots whose coordinates avoid $\tau$ are fixed by this map.)

We summarize this section with the following consequence of Theorem~\ref{T:aut}:

\begin{quotation}
The automorphism groups of the root systems $H_3$ and $H_4$ have the same connection to the Coxeter/Weyl groups $W(H_3)$ and $W(H_4)$.  In each case, half of the matroid automorphisms are geometric and half are not.  The non-geometric automorphisms arise from the $S_5$ action in $\stab(x)$ that permit odd permutations of rank-3 orthoframes in $\Pi_{15}$ planes.
\end{quotation}

\section{Geometric interpretations of $M(H_4)$.}\label{S:geo}

We can interpret the flats and orthoframes of $M(H_4)$ in terms of the 120-cell and its dual, the 600-cell.   We give the number of vertices, edges, 2-dimensional faces and 3-dimensional faces for the 120- and 600-cell  in Table~\ref{T:geoinfo} - this information  appears in Table 1(ii) of \cite{cox}.

\begin{table}[htdp]
\caption{Number of elements of the 120- and 600-cell.}
\begin{center}
\begin{tabular}{|c|c|c|c|c|} \hline
Object & Vertices & Edges & 2D faces & 3D facets \\ \hline
120-cell & 600 & 1200 & 720 & 120 \\ \hline
600-cell & 120 & 720 & 1200 & 600 \\ \hline 
\end{tabular}
\end{center}
\label{T:geoinfo}
\end{table}%

The 2-dimensional faces of the 120-cell are pentagons and the 3-dimensional facets are dodecahedra; for the 600-cell, 2-dimensional faces are triangles and 3-dimensional facets are tetrahedra.

Now each of the 60 points of the matroid corresponds to a pair of roots $\pm \bf{v}$ of the root system $H_4$. Since the roots are also the vertices of the 600-cell, we immediately get a correspondence between the points of the matroid and the pairs of opposite vertices of the 600-cell.  We can interpret the other geometric elements of the 120- and 600-cell through the matroid $M(H_4)$ in Table~\ref{T:geocorr}.  We also remark the 75  matroid orthoframes correspond to 75 embedded hypercubes in the 120-cell or 600-cell.

\begin{table}[htdp]
\caption{Correspondence between geometric elements and flats in $M(H_4)$.}
\begin{center}
\begin{tabular}{|| l c l||} \hline

Geometric Family & & Matroid Flat \\ \hline \hline
120 Vertices of 600-cell & $\Leftrightarrow$ & 60 Points \\ \hline
1200 Triangles of 600-cell  & $\Leftrightarrow$ & 600 $\Pi_3$'s  \\
720 Pentagons of 120-cell  & $\Leftrightarrow$ & 360 $\Pi_5$'s \\ \hline
600 Tetrahedra of 600-cell  & $\Leftrightarrow$ & 300 $\Pi_6$'s \\
120 Dodecahedra of 120-cell  & $\Leftrightarrow$ & 60 $\Pi_{15}$'s \\ \hline

\end{tabular}
\end{center}
\label{T:geocorr}
\end{table}%

We comment briefly on some of these connections.  For the root system $H_3$, this correspondence is explored in detail in \cite{eg}.  In that case, the roots are parallel to the edges of an icosahderon.  This makes the matroid correspondence immediate:  3-point lines of the matroid correspond to pairs of triangles in the icosahedron and  5-point lines correspond to pairs of vertices of the icosahedron (or pentagons of the dual dodecahedron).  

The chief difficulty in applying the results of \cite{eg} to $H_4$ arises from the fact that the edges of the 120-cell or 600-cell are no longer parallel to the roots.  But, since each 15-point plane is isomorphic to $M(H_3)$ as a matroid, the correspondence between dodecahedra and $\Pi_{15}$'s is clear.  We explain the connection between the 720 pentagons in the 120-cell and the 360 $\Pi_5$'s in the matroid.  In the 120-cell, a given pentagon is in two dodecahedra, but in $M(H_4)$, a given 5-point line is in 5 $\Pi_{15}$'s.  We can ``correct'' this by using $\Pi_5$'s, since each 5-point line of the matroid is in precisely five $\Pi_5$'s.

Finally, we can use the orthopoint-orthoplane bijection to get a matroidal interpretation for the 120-cell/600-cell duality.

\begin{Prop}\label{P:ptplanedual}
Let $F$ be the collection of 15-point planes, and let $B$ be the bipartite graph with vertex set $E \cup F$ with an edge joining the point $x$ to the plane $P$ if and only if $x \in P$.  Then $\Aut(B)\cong \Aut(M(H_4)) \times \Ints_2$.
\end{Prop}
\begin{proof} 
It is clear the bipartite graph $B$ allows us to reconstruct all the flats of the matroid, and any matroid automorphism acting on $E$ will necessarily be a graph automorphism of $B$.  Further, we can swap the points and the planes --  map a point $x$ to its orthoplane $P_x$. 

\end{proof}

We conclude by observing that it should be possible to treat matroids associated to other root systems (especially the exceptional $E_6, E_7,$ and $E_8$) in a coherent way that also explains the structure of those matroids.  We hope to undertake such a program in the future.


\begin{thebibliography}{99}
%
%\bibitem{om} A. Bj\"{o}rner, M. Las Vergnas, B. Sturmfels, N. White, G. Ziegler {\it Oriented Matroids}, 2nd edition, Cambridge University Press, Encyclopedia of Mathematics and its Applications, No. 46, Cambridge (1999).
%
%

\bibitem{cam}
P. J. Cameron, {\it Permutation Groups}, London Math Soc. Student Texts {\bf 45}, Cambridge Univ. Press, Cambridge (1999).

\bibitem{cs}
J. H. Conway and D. Smith, {\it On Quaternions and Octonions}, A. K. Peters, Natick, Massachusetts (2003).

\bibitem{cox} H.S.M. Coxeter, {\it Regular Polytopes}, Dover, New York (1973).

\bibitem{dut}
M Dutour Sikiric, A. Felikson and P. Tumarkin, {\it Automorphism groups of root systems matroids}, arXiv:0711.4670.

\bibitem{eg}
K. Ehly and G. Gordon,  {\it Matroid automorphisms of the root system $H_3$}, Geom. Dedicata {\bf 130}  (2007) 149-161. 

\bibitem{fggp} S. Fried, A. Gerek, G. Gordon and A. Perunicic, {\it Matroid automorphisms of the $F_4$ root system,} Electronic J. Comb. {\bf 14} (2007) R78, 12 pages.

\bibitem{fglp} L. Fern, G. Gordon, J. Leasure and S. Pronchik, {\it Matroid automorphisms and symmetry groups},  Combinatorics, Probability \& Computing {\bf 9} (2000), 105-123

\bibitem{gb} L. Grove and C. Benson, {\it Finite Reflection Groups}, second ed., Springer, New York (1985).

\bibitem{h} J. Humphreys, {\it Reflection Groups and Coxeter Groups}, Cambridge University Press, Cambridge (1990).


\bibitem{ox} J. Oxley, {\it Matroid Theory}, Oxford Graduate Texts in Mathematics, Oxford, (1993).

\end{thebibliography}
\end{document}